\documentclass{amsart}[12pt]

\usepackage{graphicx}
\usepackage{mathptmx}
\usepackage{latexsym}
\usepackage{amsmath}
\usepackage{amsfonts}
\usepackage{amssymb, amscd}
\usepackage{amsthm}

\usepackage[all]{xy}
\usepackage[square, comma, sort&compress]{natbib}

\allowdisplaybreaks
\newtheorem{theorem}{Theorem}[section]
\newtheorem{lemma}[theorem]{Lemma}
\newtheorem{proposition}[theorem]{Proposition}
\newtheorem{corollary}[theorem]{Corollary}
\theoremstyle{definition}
\newtheorem{definition}[theorem]{Definition}
\newtheorem{remark}[theorem]{Remark}
\newtheorem{example}[theorem]{Example}

\begin{document}

\title[ Notes on Cohomologies of Ternary Algebras of Associative Type]
{ Notes on Cohomologies of Ternary Algebras \\ of Associative Type}%
\author{H. Ataguema and  A. Makhlouf}%
\address{Universit\'{e} de Haute Alsace,  Laboratoire de Math\'{e}matiques, Informatique et Applications,
4, rue des Fr\`{e}res Lumi\`{e}re F-68093 Mulhouse, France}%
\email{Abdenacer.Makhlouf@uha.fr}
\email{Hammimi.Ataguema@uha.fr}

\maketitle

\begin{abstract}
The aim of this paper is to investigate the cohomologies for
 ternary algebras of associative type. We study in particular the cases of partially associative ternary algebras
 and  weak totally
associative ternary algebras. Also, we consider the Takhtajan's
construction, which was used to construct a cohomology of ternary Nambu-Lie
algebras using Chevalley-Eilenberg cohomology of Lie algebras, and  discuss it in
the case of ternary algebras of associative type. One of the main results of this paper states that a
deformation cohomology does not exist for partially associative
ternary algebras which implies that their operad is not a Koszul operad.

\end{abstract}

\section*{Introduction}
The  paper is dedicated to study cohomologies adapted to deformation
theory of  ternary algebraic structures appearing more or less
naturally in various domains of theoretical and mathematical physics
and data processing. Indeed, theoretical physics progress of quantum
mechanics and the discovery in 1973 of the Nambu mechanics (see \cite{Nambu}), as well
as a work of S. Okubo on Yang-Baxter equation (see \cite{Okubo}) gave  impulse to a
significant development on ternary algebras and more generally $n$-ary algebras.  The ternary
operations, in particular cubic matrices, were already introduced in
the nineteenth century by Cayley. The cubic matrices were considered
again and generalized by Kapranov, Gelfand and  Zelevinskii in 1994
(see \cite{KapranovGelfandZelinski}) and Sokolov in 1972 (see
\cite{Sokolov}).
 Another recent motivation to study ternary operation comes from string
theory and M-Branes where appeared naturally a so called
Bagger-Lambert algebra (see \cite{BL2007}). For other physical
applications (see \cite{Kerner, Kerner2, Kerner3, Kerner4, Kerner5, Kerner7}).

The concept of ternary algebras was introduced first by Jacobson
(see \cite{Jacobson}). In connection with problems from Jordan theory and
quantum mechanics, he defined the Lie triple systems. A Lie triple
system consists of a space of linear operators on vector space $V$
that is closed under the ternary bracket $[x,y,z]_T=[[x,y],z]$,
where $[x,y]=x y-y x$. Equivalently, a Lie triple system may be
viewed as a subspace of the Lie algebra closed relative to the
ternary product. A Lie triple system arose also in the study of
symmetric spaces (see \cite{loosSym}). We distinguish two
kinds of generalization of binary Lie algebras : ternary Lie
algebras (resp. $n$-ary Lie algebras) in which the Jacobi identity is generalized by considering
a cyclic summation over $\mathcal{S}_5$ (resp. $\mathcal{S}_{2n-1}$) instead of $\mathcal{S}_3$
(see \cite{Hanlon} \cite{Michor}), and ternary Nambu algebras (resp. $n$-ary Nambu algebras) in which
the fundamental identity generalizes the fact that the adjoint maps
are derivations. The fundamental identity appeared first in Nambu
mechanics (see \cite{Nambu}), see also  \cite{Takhtajan} for the algebraic
formulation of the Nambu mechanics.  The abstract definitions of ternary and more generally $n$-ary Nambu
algebras or $n$-ary Nambu-Lie algebras (when the bracket is skew-symmetric) were given by Fillipov in 1985 (see \cite{Filippov} in
Russian). While the  $n$-ary  Leibniz algebras
were introduced and studied in \cite{CassasLodayPirashvili}. For
deformation theory and cohomologies of ternary algebras of Lie
type, we refer to
\cite{Gautheron,Gautheron1,Takhtajan1,Kubo,Harris,Takhtajan1,Yamaguti}.

In another hand,  ternary algebras or more generally $n$-ary algebras of associative type were studied
by Carlsson, Lister, Loos (see \cite{Carlsson,Lister, Loos}). The typical and founding example of
 totally associative ternary algebra was introduced by Hestenes (see \cite{Hestenes})
  defined on the linear space of rectangular  matrices $A,B,C\in \mathcal{M}_{m,n}$ with complex entries by $AB^{\ast }C$ where $B^{\ast}$ is the conjugate transpose matrix of $B$. This operation is strictly speaking not a ternary algebra product on $\mathcal{M}_{m,n}$ as it is linear on the first and the third arguments but conjugate-linear on the second argument. The
ternary operation of associative type leads to two principal classes  : totally associative ternary algebras
 and partially associative ternary algebras.  Also they admit some variants. The
 totally associative ternary algebras are also called associative triple systems. The operads of $n$-ary algebras were studied by Gnedbaye (see \cite{Gnedbaye1,Gnedbaye2}), see also \cite{GozeRemm,Hoffbeck}.
  The
cohomology of totally associative ternary algebras was studied by
Carlsson through the embedding (see \cite{Carlsson1}). In \cite{AM2007}, we  extended to ternary algebras of associative
type, the 1-parameter formal deformations introduced by Gerstenhaber
\cite{Gerstenhaber1}, see \cite{Makhlouf} for a review. We built a 1-cohomology and 2-cohomology of
partially associative ternary algebras fitting with the deformation
theory.

The generalized Poisson structures and $n$-ary Poisson brackets were discussed in \cite{De Azcarraga,De Azcarraga2,Michor,Grabowski}. While the quantization problem was considered in \cite{Flato,Dito1}. Further generalizations and related works could be found in  \cite{Bazunova,Duplij,GozeRaush,Rausch1,Rausch2,Rausch3}.

In this paper we summarize in the first Section the definitions of
ternary algebras of associative type and Lie type with examples, and
recall  the basic settings of homological algebra. Section 2 is
devoted to study the cohomology of partially associative ternary
algebras with values in the algebra. We provide the first and the
second coboundary operators and show that their extension to a
3-coboundary does not exist. This shows that the operad of partially
ternary algebras is not a Koszul operad. In Section 3, we consider
weak totally associative ternary algebras for which we construct a
$p$-coboundary operator extending, to any $p$, the
2-coboundary operators already defined by Takhtajan
(see \cite{Takhtajan1}). In Section 4, we discuss Takhtajan's construction
for ternary algebras of associative type. The process was introduced
by Takhtajan to construct a cohomology for ternary algebras of Lie
type starting from a cohomology of binary algebras. It was used to
derive the cohomology of ternary Nambu-Lie algebras from the
Chevalley-Eilenberg cohomology of Lie algebras. We show that the
cohomology of ternary algebras of partially associative type cannot be
constructed from binary algebras of associative type. We also show in Section 5,
that the skew-associative binary algebras do not carry a cohomology
fitting with deformation theory and therefore their operad is not
Koszul as well.

\section{Generalities}

In this section we summarize the definitions of different
 ternary algebra structures of associative type and Lie type and
 provide some example, then give
general settings for cohomology theories.

\subsection{ Ternary Algebra Structures}

Let $\mathbb{K }$ be an algebraically closed field of characteristic
zero and  $V$ be a $\mathbb K$-vector space. A ternary operation on $V$ is
a linear map $ m:V\otimes V\otimes V\longrightarrow V$ or a
trilinear map $ m:V\times V\times V\longrightarrow V$. If $V$ is
$n$-dimensional vector space and $B=\{e_{1},...,e_{n}\}$ be a basis
of $V$, the
ternary operation $m$ is\ completely determined by its structure constants $%
\{C_{ijk}^{s}\}$ where $
m(e_{i}\otimes e_{j}\otimes e_{k}) =\overset{n}{\underset{s=1}{\sum }}%
C_{ijk}^{s}e_{s}. $ A ternary operation is said to be
\emph{symmetric} (resp. \emph{skew-symmetric}) if
\begin{equation}
m(x_{\sigma (1)}\otimes x_{\sigma (2)}\otimes x_{\sigma (3)})=
m(x_{1}\otimes x_{2}\otimes x_{3}), \quad\forall \sigma \in \mathcal{S}_{3}\text{ and }%
\forall x_{1}, x_{2}, x_{3}\in V
\end{equation}
resp.
\begin{equation}
m(x_{\sigma (1)}\otimes x_{\sigma (2)}\otimes x_{\sigma
(3)})=Sgn(\sigma
)m(x_{1}\otimes x_{2}\otimes x_{3}),\quad\forall \sigma \in \mathcal{S}_{3}\text{ and }%
\forall x_{1},x_{2},x_{3}\in V
\end{equation}
Where $Sgn(\sigma )$ denotes the signature of the permutation
$\sigma\in \mathcal{S}_{3}$.

We have the following type of "associative" ternary operations.

\begin{definition}
A totally associative ternary algebra is given by a $\mathbb{K}$-vector
space $V$\ and a ternary operation $m$ satisfying, for every $x_1,\cdots,x_5\in V$%
,
\begin{equation}
m(m(x_1\otimes x_2\otimes x_3)\otimes x_4\otimes x_5) =m(x_1\otimes m(x_2\otimes x_3\otimes x_4)\otimes x_5) =m(x_1\otimes x_2\otimes m(x_3\otimes x_4\otimes x_5))
\end{equation}
\end{definition}
\begin{example}
Let $\{e_1,e_2\}$ be a basis of a 2-dimensional space  $V=\mathbb{K
}^2$, the ternary  operation on $V$ given by
\begin{equation*}
\begin{array}{ll}
\begin{array}{lll}
m(e_{1}\otimes  e_{1}\otimes  e_{1}) & = & e_{1} \\
m(e_{1}\otimes  e_{1}\otimes  e_{2}) & = & e_{2} \\
m(e_{1}\otimes  e_{2}\otimes  e_{2}) & = & e_{1}+e_{2} \\
m(e_{2}\otimes  e_{1}\otimes  e_{1}) & = & e_{2}%
\end{array}
&
\begin{array}{lll}
m(e_{2}\otimes  e_{2}\otimes  e_{1}) & = & e_{1}+e_{2} \\
m(e_{2}\otimes  e_{2}\otimes  e_{2}) & = & e_{1}+2e_{2} \\
m(e_{1}\otimes  e_{2}\otimes  e_{1}) & = & e_{2} \\
m(e_{2}\otimes  e_{1}\otimes  e_{2}) & = & e_{1}+e_{2}%
\end{array}%
\end{array}
\end{equation*}

defines a totally associative ternary algebra.
\end{example}
\begin{definition}
A weak totally associative ternary algebra is given by a
$\mathbb{K}$-vector space $V$\ and a ternary operation $m$,
satisfying for every
$x_1,\cdots,x_5\in V$%
,
\begin{equation}
m(m(x_1\otimes x_2\otimes x_3)\otimes x_4\otimes x_5) =m(x_1\otimes x_2\otimes m(x_3\otimes x_4\otimes x_5))
\end{equation}
\end{definition}

Naturally, any totally associative ternary algebra is a weak totally
associative ternary algebra.

\begin{definition}
A partially associative ternary algebra is given by a $\mathbb{K}$-vector
space $V $ and a ternary operation $m$ satisfying, for every $x_1,\cdots,x_5\in V$,
\begin{equation}
m(m(x_1\otimes x_2\otimes x_3)\otimes x_4\otimes x_5) +m(x_1\otimes m(x_2\otimes x_3\otimes
x_4)\otimes x_5) +m(x_1\otimes x_2\otimes m(x_3\otimes x_4\otimes x_5))=0
\end{equation}
\end{definition}

\begin{example}
Let $\{e_1,e_2\}$ be a basis of a 2-dimensional space
  $V=\mathbb{K }^2$, the ternary  operation on $V$ given by
$
m(e_{1}\otimes  e_{1}\otimes  e_{1})=e_{2}
$
defines a partially  associative ternary algebra.
\end{example}

We introduce in the following some variants of partial total
associativity of ternary operations.
\begin{definition}
An alternate partially associative ternary algebra of first kind is given by a $\mathbb{K}$%
-vector space $V$ and a ternary operation $m$ satisfying, for every $x_1,\cdots,x_5\in V$,
\begin{equation}
m(m(x_1\otimes x_2\otimes x_3)\otimes x_4\otimes x_5) -m(x_1\otimes m(x_2\otimes x_3\otimes
x_4)\otimes x_5) +m(x_1\otimes x_2\otimes m(x_3\otimes x_4\otimes x_5))=0
\end{equation}
The alternate partially associative ternary algebra is of second
kind it satisfies :
\begin{equation}
m(m(x_1\otimes x_2\otimes x_3)\otimes x_4\otimes x_5) -m(x_1\otimes m(x_2\otimes x_3\otimes
x_4)\otimes x_5) -m(x_1\otimes x_2\otimes m(x_3\otimes x_4\otimes x_5))=0
\end{equation}
\end{definition}

\begin{remark}
Let $(V,\cdot)$ be a bilinear associative algebra. Then, the ternary
operation, defined by
$
m(x_1\otimes x_2\otimes x_3)=(x_1\cdot x_2\cdot x_3)
$
determines on the vector space $V$ a structure of totally associative
ternary algebra which is not partially associative.
\end{remark}

\begin{definition}
A ternary operation $m$ is said to be commutative if
\begin{equation}
\sum_{\sigma \in \mathit{S}_{3}}{\ Sgn(\sigma )
m(x_{\sigma(1)}\otimes x_{\sigma (2)}\otimes x_{\sigma (3)}) =0},\
\quad \forall x_{1},x_{2},x_{3}\in V
\end{equation}

\end{definition}

\begin{remark}
A symmetric ternary operation is commutative.
\end{remark}
In the following, we recall the definitions of ternary algebras of
Lie type.
\begin{definition}
A ternary Lie algebras is a skew-symmetric ternary operation $[~,~,~]$ on a $\mathbb K$-vector space $V$ satisfying $\forall x_{1},\cdots, x_{5}\in V$ the
following generalized Jacobi condition
\begin{align*}
\sum_{\sigma \in \mathit{S}_{3}}{\ Sgn(\sigma )[[x_{\sigma
(x_{1})},x_{\sigma (x_{2})},x_{\sigma (x_{3})}],
x_{\sigma(x_{4})},x_{\sigma (x_{5})}]} =0
\end{align*}
\end{definition}

As in the binary case, there is a functor which makes correspondence to any
partially associative ternary algebra a ternary Lie algebra (see \cite
{Gnedbaye1,Gnedbaye2}).
\begin{proposition}To any partially associative ternary algebra on a
vector space $V$ with ternary operation $m$, one associates a
ternary Lie algebra on $V$ defined $\forall x_{1},x_2, x_{3}\in V$ by the bracket
\begin{equation}
[x_{1},x_{2},x_{3}] =\underset{\sigma \in \mathit{S}_{3}}{\sum Sgn(}\sigma
)m(x_{\sigma (1)}\otimes x_{\sigma (2)}\otimes x_{\sigma (3)})
\end{equation}

\end{proposition}
There is another kind of  ternary  algebras of Lie type, they are
called ternary Nambu  algebra. They appeared naturally in Nambu
 mechanics which is a generalization of classical mechanics.
\begin{definition}
A ternary Nambu algebra is a ternary bracket on a $\mathbb K$-vector space $V$ satisfying a so-called
fundamental or Filippov  identity :
\begin{equation*}
[ x_{1,}x_{2,}[x_{3,}x_{4,}x_{5}]
=[[x_{1,}x_{2,}x_{3}],x_{4,}x_{5}]+[x_{3,}[x_{1,}x_{2,}x_{4}]_{,}x_{5}]+
[ x_{3,}x_{4,}[x_{1,}x_{2,}x_{5}]]
\end{equation*}
$\forall x_{1,}\cdots, x_{5}\in V.$

When the bracket is skew-symmetric the ternary algebra is called ternary Nambu-Lie algebra.

The Lie triple system is defined as a vector space $V$ equipped with a ternary bracket that satisfies the same fundamental identity as a Nambu-Lie bracket but instead of skew-symmetry, it satisfies the condition
\begin{equation*}
[ x_{1},x_{2},x_{3}]+[ x_{2},x_{3},x_{1}]+[ x_{3},x_{1},x_{2}]=0
\end{equation*}
\end{definition}
\begin{example}
The polynomial algebra of  $3$ variables $x_{1},x_{2},x_{3},$ endowed with a ternary operation defined by the functional Jacobian :
\begin{equation*}
\lbrack f_{1},f_{2},f_{3}]=\left\vert
\begin{array}{ccc}
\frac{\delta f_{1}}{\delta x_{1}} & \frac{\delta f_{1}}{\delta x_{2}} &
\frac{\delta f_{1}}{\delta x_{3}} \\
\frac{\delta f_{2}}{\delta x_{1}} & \frac{\delta f_{2}}{\delta x_{2}} &
\frac{\delta f_{2}}{\delta x_{3}} \\
\frac{\delta f_{3}}{\delta x_{1}} & \frac{\delta f_{3}}{\delta x_{2}} &
\frac{\delta f_{3}}{\delta x_{3}}%
\end{array}%
\right\vert
\end{equation*}%

is a ternary Nambu-Lie algebra.
\end{example}
We have also this fundamental example :
\begin{example}

Let $V=\mathbb{R}^4$ be the 4-dimensional oriented euclidian space over $\mathbb{R}$. The bracket of 3 vectors $\overset{\rightarrow }{x_{1}},\overset{\rightarrow }{x_{2}},\overset%
{ \rightarrow }{x_{3}}$   is given by :
\begin{equation*}
\lbrack \overset{\rightarrow }{x_{1}},\overset{\rightarrow }
{x_{2}},\overset{\rightarrow }{x_{3}} ]=\overset{\rightarrow
}{x_{1}}\times \overset{\rightarrow }{x_{2}} \times
\overset{\rightarrow }{x_{3}}=\left\vert
\begin{array}{c}
\begin{array}{cccc}
x_{11} & x_{12} & x_{13} & \overset{\rightarrow }{e_{1}} \\
x_{21} & x_{22} & x_{23} & \overset{\rightarrow }{e_{2}}%
\end{array}
\\
\begin{array}{cccc}
x_{31} & x_{32} & x_{33} & \overset{\rightarrow }{e_{3}}%
\end{array}
\\
\begin{array}{cccc}
x_{41} & x_{42} & x_{43} & \overset{\rightarrow }{e_{4}}%
\end{array}%
\end{array}%
\right\vert
\end{equation*}
where $(x_{1r},...,x_{4r})_{r=1,2,3}$ are the coordinates of  $%
\overset{\rightarrow }{x_{r}}$ with respect to orthonormal basis $\{e_{r}\}$.
Then,  $(V,[.,.,.])$ is a ternary Nambu-Lie algebra.
\end{example}
\begin{remark}
Every ternary Nambu-Lie algebra on $\mathbb{R}^4$ could be deduced from the previous example (see \cite{Gautheron1}).
\end{remark}
\subsection{Homological Algebra of Ternary Algebras}

The basic concepts of homological algebra are those of a complex and
homomorphisms of complexes, defining the category of complexes, see for example \cite{weibel}. A \emph{chain complex} $\mathcal{C} _{.}$ is a sequence $\mathcal{C} =\{\mathcal{C}
_{p}\}_{p}$ of abelian groups or more generally objects of an abelian category and an indexed set $%
\delta=\{\delta_{p}\}_{p}$ of homomorphisms $\delta_{p}:\mathcal{C}
_{p}\rightarrow \mathcal{C} _{p-1}$ such that $\delta_{p-1}\circ \delta_{p}=0$ for all $p$. A chain complex can be considered as a cochain complex by reversing
the enumeration $\mathcal{C} ^{p}=\mathcal{C}
_{-p}$ and $\delta^{p}=\delta_{-p}$. A \emph{cochain complex} $\mathcal{C} $ is a
sequence of abelian groups and homomorphisms
$
\cdots \overset{\delta^{p-1}}{\longrightarrow }\mathcal{C} ^{p}\overset{\delta^{p}}{%
\longrightarrow }\mathcal{C} ^{p+1}\overset{\delta^{p+1}}{\longrightarrow }\cdots
$
with the property $\delta^{p+1}\circ \delta^{p}=0$ for all $p$.

The homomorphisms $\delta^p$ are called \emph{coboundary operators} or \emph{codifferentials}.
 A \emph{cohomology} of a cochain complex $\mathcal{C} $ is given by  the groups $%
H^{p}(\mathcal{C} )=Ker\delta^{p}/Im\delta^{p-1}$.

The elements of $\mathcal{C} ^{p}$ are $p$-cochains, the
elements of $Z^{p}:=Ker \delta^{p}$ are $p$-cocycles, the elements of $%
B^{p}:=Im\delta^{p-1}$ are $p$-coboundaries. Because
$\delta^{p+1}\circ \delta^{p}=0$ for all $p$, we have $0\subseteq
B^{p}\subseteq Z^{p} \subseteq \mathcal{C} ^{p}$ for all $p$. The
$p^{th}$ cohomology group is the quotient $H^p=Z^{p}/B^{p}$.

We introduce in the following the $p$-cochains for a ternary algebra
of associative type $\mathcal{A}=(V,m)$.

\begin{definition}
We call  $p$-\emph{cochain} of a ternary algebra $\mathcal{A}=(V,m)$ a linear map
\begin{equation*}
\varphi :V^{\otimes 2p+1}\longrightarrow V
\end{equation*}

The $p-$cochains set on $V$ is
$
\ \mathcal{C} ^{p}({\mathcal{A} ,\mathcal{A} })=\{\varphi :V^{\otimes 2p+1}=\underset{2p+1%
\text{ }times}{\underbrace{V\otimes V\otimes ...\otimes V}}\longrightarrow
V\}
$
\end{definition}

\begin{remark}
The set $\mathcal{C} ^{p}(\mathcal{A},\mathcal{A})$ is an abelian group.
\end{remark}

We define a circle operation on cochains as usual, that is a map
\begin{equation*}\circ \ : \ \mathcal{C} ^{p}(\mathcal{A},\mathcal{A}) \times
\mathcal{C} ^{q}(\mathcal{A},\mathcal{A})\longrightarrow \mathcal{C}
^{p+q}(\mathcal{A},\mathcal{A}) \quad
(\varphi,\psi)\longmapsto\varphi\circ\psi
\end{equation*} such that
\begin{equation*}\varphi\circ\psi(x_1\otimes \cdots\otimes x_{2p+2q+1})=
\sum_{i=0}^{2p}{\varphi(x_1\otimes \cdots\otimes \psi(x_{i+1}\otimes
\cdots\otimes x_{i+2q+1})\otimes  \cdots\otimes x_{2p+2q+1})}
\end{equation*}
%\textbf{Gerstenhaber algebra ???}

One has a cochain complex for  ternary algebras $\mathcal{A}$ with values in $\mathcal{A}$ if there exists a sequence
of abelian groups and homomorphisms
$
\cdots \overset{\delta ^{p-1}}{\longrightarrow }\mathcal{C} ^{p}(\mathcal{A},\mathcal{A})\overset{%
\delta ^{p}}{\longrightarrow }\mathcal{C} ^{p+1}(\mathcal{A},\mathcal{A})\overset{\delta ^{p+1}}{%
\longrightarrow }\cdots
$
such that for all $p$,  $\delta ^{p+1}\circ \delta ^{p}=0$.

\section{Cohomology of Partially Associative Ternary Algebras}

We have studied in \cite{AM2007} deformations of partially
associative ternary algebras which are intimately linked to
cohomology groups. We have introduced the operators $\delta^{1}$ and
$\delta^{2}$ which should correspond to a complex of partially
associative ternary algebra defining a deformation
cohomology. In the following we recall the definitions of $\delta^{1}$ and $%
\delta^{2}$ and show that it is impossible to extend these operators
to an operator $\delta^{3}$.   As a consequence, we deduce that the operad of partially associative ternary algebras is not Koszul, see \cite{markl,Ginzburg} about Koszulity.

Let ${\mathcal{A} }=(V,m)$ be a partially associative ternary algebra on a $%
\mathbb{K}$-vector space $V.$

\begin{definition}
We call ternary $1$-coboundary operator the map
\begin{equation*}
\delta ^{1}: \mathcal{C} ^{0}({\mathcal{A} ,\mathcal{A} }) \longrightarrow
\mathcal{C} ^{1}({\mathcal{A} ,\mathcal{A} }), \quad f \longmapsto \delta ^{1}f
\end{equation*}
defined by
\begin{equation*}
\delta ^{1}f(x_1\otimes x_2\otimes x_3) =f(m(x_1\otimes x_2\otimes
x_3))-m(f(x_1)\otimes x_2\otimes x_3)- m(x_1\otimes f(x_2)\otimes
x_3))-m(x_1\otimes x_2\otimes f(x_3))
\end{equation*}
\end{definition}

\begin{definition}
We call ternary $2$-coboundary operator  the map
\begin{equation*}
\delta^{2}: \mathcal{C}^{1}({\mathcal{A} ,\mathcal{A} }) \longrightarrow \mathcal{C}^{2}({%
\mathcal{A} ,\mathcal{A} }), \quad \varphi \longmapsto \delta ^{2}\varphi
\end{equation*}
defined by
\begin{eqnarray*}
\delta^{2}\varphi (x_1\otimes x_2\otimes x_3\otimes x_4\otimes x_5)
= m(\varphi (x_1\otimes x_2\otimes x_3)\otimes x_4\otimes x_5)
+ m(x_1\otimes \varphi (x_2\otimes
x_3\otimes x_4)\otimes x_5) \\
+m(x_1\otimes x_2\otimes \varphi (x_3\otimes x_4\otimes x_5))
+ \varphi
(m(x_1\otimes x_2\otimes x_3)\otimes x_4\otimes x_5) \\
+\varphi (x_1\otimes m(x_2\otimes x_3\otimes x_4)\otimes x_5)
+\varphi (x_1\otimes x_2\otimes m(x_3\otimes x_4\otimes x_5))
\end{eqnarray*}
\end{definition}

\begin{remark}
The operator $\delta^{2}$ can also be defined by
\begin{equation*}
\delta^{2}\varphi =\varphi \circ m+m\circ \varphi
\end{equation*}
%where $\circ$ is the operation defined in the paragraph (4.1.1).
\end{remark}

\begin{proposition}
We have
\begin{equation*}
\delta ^{2}\circ \delta ^{1}=0
\end{equation*}

\end{proposition}
\begin{proof}Let $f$ be a  0-cochain. We compute
 $\delta ^{2}( \delta ^{1} f)$.

We have for all $x_1,x_2,x_3, x_{4},x_{5}\in V$ :
\begin{equation*}
\begin{array}{c}
\delta ^{2}(\delta ^{1} f)(x_1\otimes x_2\otimes x_3\otimes  x_{4}\otimes x_{5})=\\
m(f(m(x_{1}\otimes x_{2}\otimes  x_{3}))\otimes  x_{4}\otimes
x_{5})-m
(m(f(x_{1})\otimes  x_{2}\otimes  x_{3})\otimes  x_{4}\otimes x_{5})-\\
m(m(x_{1}\otimes  f(x_{2})\otimes  x_{3})\otimes  x_{4}\otimes
x_{5})-m
(m(x_{1}\otimes  x_{2}\otimes f(x_{3}))\otimes  x_{4}\otimes x_{5})+\\
m(x_{1}\otimes  f(m(x_{2}\otimes x_{3}\otimes x_{4}))\otimes
x_{5})-m(x_{1}\otimes m(f(x_{2})\otimes  x_{3}\otimes x_{4})\otimes
x_{5})-\\m(x_{1}\otimes m(x_{2}\otimes  f(x_{3})\otimes  x_{4})
\otimes x_{5})-m(x_{1}\otimes m(x_{2}\otimes  x_{3}\otimes f(x_{4}))\otimes x_{5})+ \\
m(x_{1}\otimes  x_{2}\otimes (f(m(x_{3}\otimes x_{4}\otimes
x_{5}))))-m(x_{1}\otimes  x_{2}\otimes m(f(x_{3})\otimes
x_{4}\otimes x_{5})) -\\m(x_{1}\otimes  x_{2}\otimes m(x_{3}\otimes
f(x_{4})\otimes  x_{5}))-m(x_{1}\otimes  x_{2}\otimes m(x_{3}\otimes  x_{4}\otimes f(x_{5})))+\\
f(m(m(x_{1}\otimes  x_{2}\otimes  x_{3})\otimes x_{4}\otimes
x_{5}))-m(f(m(x_{1}\otimes  x_{2}\otimes  x_{3}))\otimes
x_{4}\otimes x_{5})-\\m(m(x_{1}\otimes  x_{2}\otimes  x_{3})\otimes
f(x_{4})\otimes  x_{5})-m (m(x_{1}\otimes  x_{2}\otimes
x_{3})\otimes x_{4}\otimes f(x_{5}))+\\f(m(x_{1}\otimes m
(x_{2}\otimes x_{3}\otimes x_{4})\otimes x_{5}))-m(f(x_{1})\otimes
m(x_{2}\otimes  x_{3}\otimes x_{4}) \otimes  x_{5})-\\m(x_{1}\otimes
f(m(x_{2}\otimes  x_{3}\otimes x_{4}))\otimes  x_{5})-m(x_{1}
\otimes  m(x_{2}\otimes  x_{3}\otimes x_{4})\otimes f(x_{5}))+\\
f(m(x_{1}\otimes x_{2}\otimes  m(x_{3}\otimes x_{4}\otimes
x_{5})))-m(f(x_{1})\otimes  x_{2}\otimes  m(x_{3}\otimes
x_{4}\otimes x_{5}))-\\m(x_{1}\otimes  f(x_{2})\otimes
m(x_{3}\otimes x_{4}\otimes  x_{5}))-m(x_{1}\otimes x_{2}\otimes
f(m(x_{3}\otimes  x_{4}\otimes x_{5})))=0
\end{array}%
\end{equation*}
\end{proof}

The cohomology spaces relative to these coboundary operators are

\begin{definition}
The   $1$-cocycles space of ${\mathcal{A}}$ is
\begin{equation*}
\mathit{Z}^{1}({\mathcal{A} ,\mathcal{A} })=\{f: V\longrightarrow
V\,:\,\delta^{1}f=0\}
\end{equation*}
The   $2$-coboundaries space of ${\mathcal{A} }$ is
\begin{equation*}
\mathit{B}^{2}({\mathcal{A} ,\mathcal{A} }) =\{\varphi :V^{\otimes
3}\longrightarrow V\,:\, \varphi =\delta ^{1}f,f\in \mathcal{C}
^{0}({\mathcal{A} },{\mathcal{A} })\}
\end{equation*}
The   $2$-cocycles space of ${\mathcal{A}}$ is
\begin{equation*}
\mathit{Z}^{2}({\mathcal{A} ,\mathcal{A} }) =\{f:V^{\otimes 3}\longrightarrow
V\,:\,\delta ^{2}f=0\}
\end{equation*}
\end{definition}

\begin{remark}
One has $\mathit{B}^{2}({\mathcal{A} ,\mathcal{A} }){\subset }\mathit{Z}^{2}({\mathcal{A} ,\mathcal{A} }%
) $, because $\delta ^{2}\circ \delta ^{1}=0$. Note also that $\mathit{Z}%
^{1}({\mathcal{A} ,\mathcal{A} })$ corresponds to the derivations space, denoted also  $Der{(\mathcal{A} )}$,  of the partially associative
ternary algebra $%
{\mathcal{A} }$.
\end{remark}

\begin{definition}
We call the $p^{th}$ cohomology group $(p=0,1)$ of the partially associative
ternary algebra ${\mathcal{A} }$ the quotient
\begin{equation*}
\mathit{H}^{p}(\mathcal{A} ,\mathcal{A} )=\frac{\mathit{Z}^{p}(\mathcal{A} ,\mathcal{A} )}{\mathit{B}%
^{p}(\mathcal{A} ,\mathcal{A} )},\qquad p=1,2
\end{equation*}
\end{definition}

The following proposition shows that we cannot extend the 1-cohomology and 2-cohomology corresponding to the operators
$\delta ^{1}$ and $\delta ^{2}$ to a 3-cohomology.

\begin{proposition}
\label{prop} Let $\mathcal{A} =(V,m)$ be a partially associative ternary
algebra.
There is no  3-cohomology extending the 2-cohomology corresponding to the
coboundary operator
\begin{equation*}
\delta ^{2}:%
\begin{array}{ccc}
\mathcal{C} ^{1}({ \mathcal{A} ,\mathcal{A} }) & \longrightarrow & \mathcal{C} ^{2}(%
{ \mathcal{A} ,\mathcal{A} })%
\end{array}%
\end{equation*}
defined for all $x_{1}, x_{2}, x_{3}, x_{4}, x_{5}\in V$ by
\begin{eqnarray*}
\delta ^{2}\varphi (x_{1}\otimes  x_{2}\otimes  x_{3}\otimes
x_{4}\otimes x_{5})=m(\varphi (x_{1}\otimes  x_{2}\otimes
x_{3})\otimes x_{4}\otimes  x_{5})+ m(x_{1}\otimes  \varphi
(x_{2}\otimes x_{3}\otimes  x_{4})\otimes  x_{5})\\
+m(x_{1}\otimes x_{2}\otimes  \varphi (x_{3}\otimes  x_{4}\otimes
x_{5}))+ \varphi ( m(x_{1}\otimes  x_{2}\otimes  x_{3})\otimes
x_{4}\otimes  x_{5})\\  + \varphi ( x_{1}\otimes
m(x_{2}\otimes  x_{3}\otimes  x_{4})\otimes  x_{5})+ \varphi (
x_{1}\otimes  x_{2}\otimes  m(x_{3}\otimes  x_{4}\otimes
x_{5}))
\end{eqnarray*}
\end{proposition}

\begin{proof}
We consider a 3-cochain $f$, that is a map $f:V^{\otimes 5}\rightarrow V$, and set
\begin{equation*}
\begin{array}{ccc}
\delta^3 f(x_1\otimes x_2\otimes x_3\otimes x_4\otimes x_5\otimes
x_6\otimes x_7)=
\\ a_1 m(x_1\otimes x_2\otimes f(x_3\otimes x_4\otimes x_5\otimes x_6\otimes x_7))+a_2
m(x_1\otimes f(x_2\otimes x_3\otimes x_4\otimes x_5\otimes x_6)\otimes x_7)+ \\
a_3 m(f(x_1\otimes x_2\otimes x_3\otimes x_4\otimes x_5)\otimes
x_6\otimes x_7)+
a_4f(m(x_1\otimes x_2\otimes x_3)\otimes x_4\otimes x_5\otimes x_6\otimes x_7)+ \\
a_5f(x_1\otimes m(x_2\otimes x_3\otimes x_4)\otimes x_5\otimes
x_6\otimes x_7)+ a_6
f(x_1\otimes x_2\otimes m(x_3\otimes x_4\otimes x_5)\otimes x_6\otimes x_7)+ \\
a_7 f(x_1\otimes x_2\otimes x_3\otimes m(x_4\otimes x_5\otimes
x_6)\otimes x_7)+a_8 f(x_1\otimes x_2\otimes x_3\otimes x_4\otimes
m(x_5\otimes x_6\otimes x_7))
\end{array}
\end{equation*}
where $a_1,\cdots, a_8 \in \mathbb K$.

 Let $g$ be a $2$-cochain, that is a map
$f:V^{\otimes 5}\rightarrow V$. We compute $\delta^3 ( \delta^2 g)
(x_1\otimes x_2\otimes x_3\otimes x_4\otimes x_5\otimes x_6\otimes
x_7)$ and substitute $m(y_1\otimes y_2\otimes m(y_3\otimes
y_4\otimes y_5))$ by
\begin{equation*}-m(y_1\otimes m(y_2\otimes y_3\otimes
y_4)\otimes y_5)-m(m(y_1\otimes y_2\otimes y_3)\otimes y_4\otimes
y_5)\end{equation*}
 Then, we obtain
{\small{
\begin{eqnarray*}
\delta^3 ( \delta^2 g) (x_1\otimes x_2\otimes x_3\otimes x_4\otimes
x_5\otimes x_6\otimes x_7)&=&
(a_7-a_8) g(x_1\otimes x_2\otimes m(x_3\otimes m(x_4\otimes x_5\otimes x_6)\otimes x_7))\\
\ && +(a_6-a_8)
g(x_1\otimes x_2\otimes m(m(x_3\otimes x_4\otimes x_5)\otimes
x_6\otimes x_7))
\\
\  & &+(a_5 +a_8) g(x_1\otimes m(x_2\otimes x_3\otimes x_4)\otimes m(x_5\otimes x_6\otimes x_7))
\\
\  & & +(a_6 -a_7)
g(x_1\otimes m(x_2\otimes m(x_3\otimes x_4\otimes x_5)\otimes
x_6)\otimes x_7)
\\
\  & & +(a_5 -a_7) g(x_1\otimes m(m(x_2\otimes x_3\otimes
x_4)\otimes x_5\otimes x_6)\otimes x_7)
\\ \  & &
+(a_4 +a_8) g(m(x_1\otimes
x_2\otimes x_3)\otimes x_4\otimes m(x_5\otimes x_6\otimes x_7))
\\ \  & &
+(a_4 +a_7) g(m(x_1\otimes x_2\otimes x_3)\otimes m(x_4\otimes
x_5\otimes x_6)\otimes x_7)
\\ \  & &+
(a_5 -a_6) g(m(x_1\otimes m(x_2\otimes x_3\otimes x_4)\otimes x_5)\otimes x_6\otimes x_7)
\\
\  & &+
(a_4 -a_6) g(m(m(x_1\otimes x_2\otimes x_3)\otimes x_4\otimes
x_5)\otimes x_6\otimes x_7)
\\
\  & &
+(a_1 +a_8) m(x_1\otimes x_2\otimes
g(x_3\otimes x_4\otimes m(x_5\otimes x_6\otimes x_7)))
\\
\  & &  +(a_1 +a_7) m(x_1\otimes x_2\otimes g(x_3\otimes m(x_4\otimes x_5\otimes x_6)\otimes x_7))
\\\  & &
+(a_1 +a_6)
m(x_1\otimes x_2\otimes g(m(x_3\otimes x_4\otimes x_5)\otimes x_6\otimes x_7))
\\ \  & &
+(a_2 +a_7) m(x_1\otimes g(x_2\otimes x_3\otimes m(x_4\otimes
x_5\otimes x_6))\otimes x_7)
\\\  & &
+(a_2 +a_6) m(x_1\otimes g(x_2\otimes
m(x_3\otimes x_4\otimes x_5)\otimes x_6)\otimes x_7)
\\\  & &
+(a_2 +a_5) m(x_1\otimes g(m(x_2\otimes x_3\otimes x_4)\otimes x_5\otimes x_6)\otimes x_7)
\\ \  & &
+(a_5-a_1)
m(x_1\otimes m(x_2\otimes x_3\otimes x_4)\otimes g(x_5\otimes
x_6\otimes x_7))
\\\  & &
+(a_2-a_1) m(x_1\otimes m(x_2\otimes x_3\otimes g(x_4\otimes x_5\otimes x_6))\otimes x_7)
\\ \  & &
+(a_2-a_1)
m(x_1\otimes m(x_2\otimes g(x_3\otimes x_4\otimes x_5)\otimes
x_6)\otimes x_7)
\\\  & &
+(a_2-a_8) m(x_1\otimes m(g(x_2\otimes x_3\otimes x_4)\otimes x_5\otimes x_6)\otimes x_7)
\\ \  & &
+(a_7-a_8)
m(g(x_1\otimes x_2\otimes x_3)\otimes m(x_4\otimes x_5\otimes
x_6)\otimes x_7)
\\\  & &
+(a_3 +a_6) m(g(x_1\otimes x_2\otimes m(x_3\otimes
x_4\otimes x_5))\otimes x_6\otimes x_7)
\\\  & &
+(a_3 +a_5) m(g(x_1\otimes
m(x_2\otimes x_3\otimes x_4)\otimes x_5)\otimes x_6\otimes
x_7)
\\\  & &
+(a_3 +a_4) m(g(m(x_1\otimes x_2\otimes x_3)\otimes x_4\otimes
x_5)\otimes x_6\otimes x_7)
\\\  & &
+(a_4-a_1) m(m(x_1\otimes x_2\otimes
x_3)\otimes x_4\otimes g(x_5\otimes x_6\otimes x_7))
\\\  & &
+(a_4-a_1)
m(m(x_1\otimes x_2\otimes x_3)\otimes g(x_4\otimes x_5\otimes
x_6)\otimes x_7)
\\\  & &
+(a_3-a_1) m(m(x_1\otimes x_2\otimes g(x_3\otimes
x_4\otimes x_5))\otimes x_6\otimes x_7)
\\\  & &
+(a_3-a_8) m(m(x_1\otimes
g(x_2\otimes x_3\otimes x_4)\otimes x_5)\otimes x_6\otimes x_7)
\\\  & &
+(a_3-a_8) m(m(g(x_1\otimes x_2\otimes x_3)\otimes x_4\otimes
x_5)\otimes x_6\otimes x_7)
\\\  & &  =0
\end{eqnarray*}
}}
The equation is satisfied for all $x_1,x_2,x_3,x_4,x_5,x_6,x_7\in
V$ if and only if  $a_{1},\cdots,a_8$ are all equal to 0.

\end{proof}

\begin{corollary}
A deformation cohomology of  partially associative ternary algebras
doesn't exist. Then, the operad of the partially associative ternary
algebras $pAss^{(3)}$ is not a Koszul operad.
\end{corollary}

\begin{remark}
In \cite{Hoffbeck}, it is shown that the operad of totally
associative ternary algebras is Koszul because it has a
Poincar\'e-Birkhoff-Witt basis. Moreover its dual, the operad of
partially associative ternary algebras, is also Koszul when the
operations are in degree one. See also \cite{GozeRemm} for
constructions in this case and the recent preprint \cite{Remm2008Koszul}. The corollary claims that the operad is
not  a Koszul operad when the  operations are in degree zero.
\end{remark}
\begin{remark}
Using the same approach, we can show that the alternate partially
associative ternary algebras of first and second kind do not carry a
deformation cohomology as well, then their operads are not koszul
operads.
\end{remark}

\section{Cohomology of Weak Totally Associative Ternary Algebras}

In this section, we generalize to $p$-cohomology, for all $p$, the
1-cohomology and 2-cohomology of weak totally associative ternary
algebra defined by Takhtajan (see \cite{Takhtajan1}). Let
$\mathcal{A}=(V,m)$ be a weak totally associative ternary algebras on
a $\mathbb K$-vector space $V.$

The $1-$coboundary and $2-$coboundary operators for weak totally associative
ternary algebras $\mathcal{A}$ are defined as follows

\begin{definition}
A  $1-$coboundary operator of a weak totally associative ternary algebra
  $\mathcal{A}=(V,m)$ is the map
%\heartsuit
\begin{eqnarray*} \delta^{1}:  \mathcal{C} ^{1}({\mathcal{A},\mathcal{A}}) &
\longrightarrow &
\mathcal{C} ^{2}({\mathcal{A} ,\mathcal{A} }) \\
 f & \longmapsto & \delta ^{1}f%
\end{eqnarray*}
defined for all $x_1,x_2,x_3 \in V$ by
\begin{eqnarray*}
\delta^{1}f(x_1\otimes x_2\otimes x_3)&=&m(f(x_1)\otimes x_2\otimes
x_3)+m(x_1\otimes f(x_2)\otimes x_3)) \\\ & &+
m(x_1\otimes x_2\otimes f(x_3))-f(m(x_1\otimes x_2\otimes x_3))
\end{eqnarray*}

A  $2-$coboundary operator of a weak totally associative ternary algebra ${\mathcal{A} }$ is the map
\begin{equation*}
\begin{array}{cccc}
\delta^{2}: & \mathcal{C} ^{2}({\mathcal{A} ,\mathcal{A} }) & \longrightarrow &
\mathcal{C} ^{3}({\mathcal{A} ,\mathcal{A}}) \\
& \varphi & \longmapsto & \delta ^{2}\varphi%
\end{array}%
\end{equation*}
defined for all $x_1,\cdots,x_5 \in V$ by
\begin{eqnarray*}
\delta^{2}\varphi (x_1\otimes x_2\otimes x_3\otimes x_4\otimes
x_5)=m(x_1\otimes x_2\otimes \varphi (x_3\otimes x_4\otimes
x_5))-m(\varphi (x_1\otimes x_2\otimes
x_3)\otimes x_4\otimes x_5)   \\
+\varphi (x_1\otimes x_2\otimes m(x_3\otimes x_4\otimes
x_5))-\varphi ( m(x_1\otimes x_2\otimes x_3)\otimes x_4\otimes
x_5) &
\end{eqnarray*}
\end{definition}

\begin{remark}
One can easily show that $\delta ^{2}\circ \delta ^{1}=0$. Indeed,
\begin{equation*}
\begin{array}{ll}
\delta ^{2}\circ \delta ^{1}(f)(x_1\otimes x_2\otimes x_3\otimes
x_4\otimes
x_5)=m(x_1\otimes x_2\otimes m(f(x_3)\otimes x_4\otimes x_5)) &  \\+
m(x_1\otimes x_2\otimes m(x_3\otimes f(x_4)\otimes
x_5))+m(x_1\otimes
x_2\otimes m(x_3\otimes x_4\otimes f(x_5))) &  \\-
m(x_1\otimes f(x_2)\otimes m(x_3\otimes x_4\otimes
x_5))-m(x_1\otimes
f(m(x_2\otimes x_3\otimes x_4))\otimes x_5) &  \\-
m(m(x_1\otimes x_2\otimes x_3)\otimes f(x_4)\otimes
x_5)-m(m(x_1\otimes
x_2\otimes x_3)\otimes x_4\otimes f(x_5)) &  \\+
f(m(m(x_1\otimes x_2\otimes x_3)\otimes x_4\otimes
x_5))+m(f(x_1)\otimes
x_2\otimes m(x_3\otimes x_4\otimes x_5)) &  \\+
m(x_1\otimes f(x_2)\otimes m(x_3\otimes x_4\otimes
x_5))+m(x_1\otimes
x_2\otimes f(m(x_3\otimes x_4\otimes x_5))) &  \\-
f(m(x_1\otimes x_2\otimes m(x_3\otimes x_4\otimes
x_5)))-m(m(f(x_1)\otimes x_2\otimes x_3)\otimes x_4\otimes x_5) &  \\-
m(m(x_1\otimes f(x_2)\otimes x_3)\otimes x_4\otimes
x_5)-m(m(x_1\otimes
x_2\otimes f(x_3))\otimes x_4\otimes x_5) &  \\+
m(f(m(x_1\otimes x_2\otimes x_3))\otimes x_4\otimes
x_5)+f(m(x_1\otimes
m(x_2\otimes x_3\otimes x_4)\otimes x_5)) &  \\
-m(x_1\otimes m(x_2\otimes x_3\otimes f(x_4))\otimes
x_5)-m(x_1\otimes
m(x_2\otimes x_3\otimes x_4)\otimes f(x_5)) &  \\+
m(x_1\otimes m(x_2\otimes f(x_3)\otimes x_4)\otimes
x_5)-m(x_1\otimes
m(f(x_2)\otimes x_3\otimes x_4)\otimes x_5) &  \\
-m(f(x_1)\otimes m(x_2\otimes x_3\otimes x_4)\otimes x_5)=0 &
\end{array}%
\end{equation*}
\end{remark}

We introduce for weak associative ternary algebras the following
generalized coboundary map.
\begin{definition}
Let  $f$ be a  $(p-1)$-cochain of a weak associative ternary algebra
$\mathcal{A}=(V,m) $ and  $f\in \mathcal{C} ^{p-1}(\mathcal{A}, \mathcal{A})$ that is  $f:V^{\otimes
2p-1}\longrightarrow V$ . We set
\begin{eqnarray*}
\delta ^{p}f\left( x_{1}\otimes...\otimes x_{2p+1}\right) &=&m\left(
x_{1}\otimes x_{2}\otimes f\left( x_{3}\otimes ...\otimes
x_{2p+1}\right) \right) +  \\ \ & & \sum_{i=1}^{p}\left( -1\right)
^{i}f\left( x_{1}\otimes ...\otimes m\left( x_{2i-1}\otimes
x_{2i}\otimes x_{2i+1}\right) \otimes ...\otimes x_{2p+1}\right) +
 \\\ & &
\left( -1\right) ^{p+1}m\left( f\left( x_{1}\otimes ...\otimes
x_{2p-1}\right)
\otimes x_{2p}\otimes x_{2p+1}\right) %
\end{eqnarray*}%
\end{definition}
In particular we have
\begin{equation*}
\begin{array}{ll}
\delta ^{3}\varphi (x_{1}\otimes x_{2}\otimes x_{3}\otimes x_{4}\otimes
x_{5}\otimes x_{6}\otimes x_{7})=m(\varphi (x_{1}\otimes x_{2}\otimes
x_{3}\otimes x_{4}\otimes x_{5})\otimes x_{6}\otimes x_{7})- &  \\
\varphi ( m(x_{1}\otimes x_{2}\otimes x_{3})\otimes x_{4}\otimes
x_{5}\otimes x_{6}\otimes x_{7})+\varphi ( x_{1}\otimes x_{2}\otimes
m(x_{3}\otimes x_{4}\otimes x_{5})\otimes x_{6}\otimes x_{7})- &  \\
\varphi ( x_{1}\otimes x_{2}\otimes x_{3}\otimes x_{4}\otimes
m(x_{5}\otimes x_{6}\otimes x_{7}))+m(x_{1}\otimes x_{2}\otimes
\varphi (x_{3}\otimes x_{4}\otimes x_{5}\otimes x_{6}\otimes x_{7})) & %
\end{array}%
\end{equation*}
\begin{proposition}
We have  $ \delta ^{p+1}\circ \delta ^{p}=0$ for all $p\geq1$.
\end{proposition}

\begin{proof}
We have $\delta ^{2}\circ \delta ^{1}=0.$  Assume  $\delta ^{p}\circ
\delta ^{p-1}=0.$ We have to show that $ \delta ^{p+1}\circ \delta
^{p}=0. $

Let $\varphi$ be a $p$-cochain and $x_{1},
..., x_{2p+3}\in V$.
\begin{equation*}
\begin{array}{c}
\delta ^{p}\varphi(x_{1}\otimes  ...\otimes  x_{2p+1})= \\m(x_{1}\otimes  x_{2}\otimes
\varphi (x_{3}\otimes ...\otimes  x_{2p+1}))+(-1)^{p+1}m(\varphi (x_{1}\otimes  ...\otimes
x_{2p-1})\otimes x_{2p}\otimes  x_{2p+1}) +\\ \underset{i=1}{\overset{p}{\sum
}}(-1)^{i}\varphi (x_{1}\otimes  ...\otimes  m(x_{2i-1}\otimes  x_{2i}\otimes  x_{2i+1})\otimes
...\otimes  x_{2p+1})
\end{array}%
\end{equation*}
Then
$ \delta ^{p+1} (\delta ^{p}\varphi)
(x_{1}\otimes  ...\otimes  x_{2p+3})\ $ vanishes. Indeed
{\small{
\begin{flalign*}
& \delta ^{p+1} (\delta ^{p}\varphi)
(x_{1}\otimes  ...\otimes  x_{2p+3})\ =\\
& m(x_{1}\otimes  x_{2}\otimes  \delta ^{p}\varphi (x_{3}\otimes
...\otimes  x_{2p+3}))+(-1)^{p+2}m(\delta ^{p}\varphi (x_{1}\otimes  ...\otimes  x_{2p+1})\otimes  x_{2p+2}\otimes  x_{2p+3})+ \\
& \underset{k=1}{\overset{p+1}{\sum }}(-1)^{k}\delta ^{p}\varphi
(x_{1}\otimes  ...\otimes  m(x_{2k-1}\otimes  x_{2k}\otimes  x_{2k+1})\otimes  ...\otimes
x_{2p+3})\\
& =m(x_{1}\otimes  x_{2}\otimes  m(x_{3}\otimes x_{4}\otimes \varphi
(x_{5}\otimes ...\otimes x_{2p+3})))+\\
& (-1)^{p+1}m(x_{1}\otimes
x_{2}\otimes m(\varphi(x_{3}\otimes  ...\otimes  x_{2p+1})\otimes x_{2p+2}\otimes x_{2p+3})) +
\\
& \underset{i=1}{\overset{p}{\sum
}}(-1)^{i}m(x_{1}\otimes x_{2}\otimes \varphi(x_{3}\otimes ...\otimes m( x_{2i+1}\otimes
x_{2i+2}\otimes x_{2i+3})\otimes  ...\otimes x_{2p+3}))+
\\
& (-1)^{p+2}m(m(x_{1}\otimes
x_{2}\otimes m(\varphi(x_{3}\otimes  ...\otimes  x_{2p+1})\otimes x_{2p+2}\otimes x_{2p+3}))+
\\
& (-1)^{2p+3}m(m(\varphi(x_{1}\otimes  ...\otimes x_{2p-1})\otimes x_{2p}\otimes
x_{2p+1})\otimes x_{2p+2}\otimes x_{2p+3})+
 \\
&
 (-1)^{p+2}\underset{i=1}{\overset{p}{\sum }}(-1)^{i}m(\varphi(x_{1}\otimes ...\otimes m(x_{2i-1}\otimes x_{2i}\otimes x_{2i+1})\otimes ...\otimes x_{2p+1})\otimes  x_{2p+2}\otimes x_{2p+3})+
 \\
&
\underset{k=1}{\overset{p+1}{\sum
}}(-1)^{k}m(x_{1}\otimes x_{2}\otimes \varphi(x_{3}\otimes ...\otimes m(x_{2k-1}\otimes x_{2k}\otimes x_{2k+1})\otimes ...\otimes x_{2p+3}))+
\\
&
\underset{k=1}
{\overset{p+1}{\sum }}(-1)^{k}(-1)^{p+1}m(\varphi(x_{1}\otimes ...\otimes m(x_{2k-1}\otimes x_{2k}\otimes x_{2k+1})\otimes ...\otimes x_{2p+1})\otimes x_{2p+2}\otimes x_{2p+3})+
\\
&
\underset{k=3}{\overset{p+1}{\sum }}\underset{i=1}{\overset{k-2}{\sum }}(-1)^{k+i}\varphi (x_{1}\otimes ...\otimes m(x_{2i-1}\otimes x_{2i}\otimes x_{2i+1})\otimes ...\otimes m(x_{2k-1}\otimes x_{2k}\otimes x_{2k+1})\otimes ...\otimes x_{2p+3})+
\\
&
\underset{k=1}{\overset{p}{\sum }}(-1)^{2k}\varphi (x_{1}\otimes ...\otimes m(m(x_{2k-1}\otimes x_{2k}\otimes x_{2k+1})\otimes x_{2k+2}\otimes x_{2k+3})\otimes ...\otimes x_{2p+3})+
\\
&
\underset{k=1}{\overset{p}{\sum }}(-1)^{2k+1}\varphi (x_{1}\otimes ...\otimes m(x_{2k-1}\otimes x_{2k}\otimes m(x_{2k+1}\otimes x_{2k+2}\otimes x_{2k+3}))\otimes ...\otimes x_{2p+3})+
\\
&
\underset{k=1}{\overset{p-2}{\sum
}}\underset{i=2k+2}{\overset{p+1}{\sum }}(-1)^{k+i+1}\varphi
(x_{1}\otimes ...\otimes m(x_{2k-1}\otimes x_{2k}\otimes x_{2k+1})\otimes ...\otimes m(x_{2i-1}\otimes x_{2i}\otimes x_{2i+1})\otimes ...\otimes x_{2p+3})
\\
&
=0
\end{flalign*}
}}
\end{proof}

\section{Takhtajan's Construction}

In this section, we aim to extend to ternary algebras of associative type a process introduced by Takhtajan to construct a complex of ternary
algebras  starting from a complex of
binary algebras (see \cite{Takhtajan1}). Let $(V,m)$ be  a ternary algebra  of a given
type. We associate to it  a binary algebra on $%
W=V\otimes V$ and a map  $\Delta$.
Assume that $(\mathcal{C} ,\delta )$ is a complex for the ternary algebras and   $(M,d)$ be a complex for the binary algebras.

We define a map $\Delta$ such that  $\Delta_p$  associates to any
$p$-cochain on $V$ a $p$-cochain on $W$. It  is defined   by
\begin{equation*}
\begin{array}{cccc}
\Delta _{0} & :\mathcal{C} ^{0}=Hom(V,V) & \longrightarrow &
M^{0}=Hom(W,W)
\\
& f & \longmapsto & \Delta _{0}(f)%
\end{array}%
\end{equation*}%
such that for example  $\Delta _{0}(f)(x_{1}\otimes x_{2})=f(x_{1})\otimes
x_{2}+\alpha \ x_{1}\otimes f(x_{2}),$ \ \  $\forall x_{i}\in V$  with $\alpha\in \mathbb{K}$.

One extends this operation to
\begin{equation*}
\begin{array}{cccc}
\Delta_p & :\mathcal{C} ^{p}=Hom(V^{\otimes 2p+1},V) &
\longrightarrow &
M^{p}=Hom(W^{p+1},W) \\
& \varphi & \longmapsto & \Delta_p \varphi %
\end{array}%
\end{equation*}
defined for example, using the remark that
$
W^{\otimes p+1}\cong V^{\otimes 2p+2},
$
 by
\begin{equation*}
\Delta_p \varphi (y_{1}\otimes ...\otimes y_{2p+2})=\varphi
(y_{1}\otimes ...\otimes y_{2p+1})\otimes y_{2p+2}+\alpha \  y_{1}\otimes
\varphi (y_{2}\otimes ...\otimes y_{2p+2})
\end{equation*}

Let us assume  that one has a complex  $(M,d)$ :

\begin{equation*}
\cdots\longrightarrow M^{p-1}\overset{d^{p-1}}{\longrightarrow }M^{p}\overset{d^{p}}{%
\longrightarrow }M^{p+1}\longrightarrow\cdots
\end{equation*}

i.e. for all $p$,
$
\ d^{p}\circ d^{p-1}=0.
$

Consider for any $p>0,$ the linear maps $ \delta ^{p}:\mathcal{C}
^{p}\longrightarrow \mathcal{C} ^{p+1}\text{ } $
 satisfying
\begin{equation*}
\Delta_{p+1} \circ \delta ^{p}=d^{p}\circ \Delta_p ,\text{ \
}\forall p
\end{equation*}

The equality is well defined.

Indeed, one has for $p\geq1$
\begin{equation*}
\mathcal{C} ^{p}\overset{\Delta_p }{\longrightarrow }M^{p}\overset{d^{p}}{%
\longrightarrow }M^{p+1}
\quad
\text{ and }
\quad
\mathcal{C} ^{p}\overset{\delta ^{p}}{\longrightarrow }\mathcal{C} ^{p+1}%
\overset{\Delta_{p+1} }{\longrightarrow }M^{p+1}
\end{equation*}

\begin{lemma}
Let $p>1$. If $d^{p}\circ d^{p-1}=0$ \ then  $\ \delta ^{p}\circ
\delta ^{p-1}=0$.
\end{lemma}

\begin{proof}
One has $\Delta_{p+1} \circ \delta ^{p}=d^{p}\circ \Delta_{p}$, then
$\Delta_{p+1} \circ \delta ^{p}\circ \delta ^{p-1}=d^{p}\circ
\Delta_{p} \circ \delta ^{p-1}=d^{p}\circ d^{p-1}\circ \Delta_{p-1} =0$, because $%
d^{p}\circ d^{p-1}=0.$
\end{proof}

As a consequence of the previous lemma, one may obtain a complex of
ternary algebras starting from a complex of binary algebras and a
map $\Delta$. This process was used by Takhtajan to construct a
cohomology of ternary Nambu algebras using the
Chevalley-Eilenberg cohomology of Lie algebras. The binary multiplication  used to that end is
defined as follows :

Let $(V,[~,~,~])$ be a ternary Nambu algebra.  Set $W=V \otimes V$.
The multiplication on $W$ is defined for
$x_1 \otimes x_2,y_1 \otimes y_2\in W$ by
\begin{equation*}
[x_1 \otimes x_2,y_1 \otimes y_2]_W=[x_1 , x_2,y_1] \otimes y_2+y_1 \otimes[x_1 , x_2, y_2]
\end{equation*}

\subsection{Takhtajan's Construction and Ternary Algebras of Associative Type}
In the  sequel we show that we cannot derive  a cohomology of  a partially associative
ternary algebra  from a cohomology of
binary algebras of associative type. A construction is possible in the case of totally associative ternary algebras
but the cohomology obtained is the cohomology of weak totally associative ternary algebras described above.

 A binary algebra is called of associative type if it is given by a vector
 space
$V$ and a multiplication $\mu$ satisfying an identity of the form
\begin{equation*}\mu(\mu(u\otimes v)\otimes w)+ \lambda\  \mu(u\otimes \mu(v\otimes w))=0
\end{equation*}
where $\lambda$ is a scalar element different from zero. In
particular, we have associative algebras for $\lambda=-1$ and
skew-associative algebras for $\lambda=1$. In the last section, we show that the skew-associative algebras cannot carry a cohomology adapted to deformation theory.

In the following, we try to adapt the Takhatjan's procedure to ternary algebras of associative type.
We set
\begin{equation*}
\begin{array}{c}
\mu ((x_{1}\otimes x_{2})\otimes (y_{1}\otimes y_{2})) =
m(x_{1}\otimes x_{2}\otimes y_{1})\otimes y_{2}+\alpha\
x_{1}\otimes m(x_{2}\otimes y_{1}\otimes y_{2})
\end{array}%
\end{equation*}
where $\alpha\in \mathbb{K}$.

In order to check whether $\mu$ is a binary
operation of associative type, we compute :
\begin{eqnarray*}
A_{1}&=&\mu (\mu ((x_{1}\otimes x_{2})\otimes (y_{1}\otimes
y_{2}))\otimes
(z_{1}\otimes z_{2}))\\
\ &=&\mu ((m(x_{1}\otimes x_{2}\otimes y_{1})\otimes y_{2})\otimes
(z_{1}\otimes z_{2})) +\alpha\ \mu ((x_{1}\otimes m(x_{2}\otimes
y_{1}\otimes y_{2}))\otimes (z_{1}\otimes z_{2}))
\\\ & =& m(m(x_{1}\otimes x_{2}\otimes y_{1})\otimes y_{2}\otimes
z_{1})\otimes z_{2} +\alpha \ m(x_{1}\otimes x_{2}\otimes
y_{1})\otimes m(y_{2}\otimes z_{1}\otimes z_{2})
\\\  & &+\alpha \ m(x_{1}\otimes m(x_{2}\otimes y_{1}\otimes y_{2})\otimes z_{1})\otimes z_{2}
+\alpha ^{2}\ x_{1}\otimes m(m(x_{2}\otimes y_{1}\otimes
y_{2})\otimes z_{1}\otimes z_{2})
\end{eqnarray*}
and
\begin{eqnarray*}
A_{2}&=& \mu ((x_{1}\otimes x_{2})\otimes \mu ((y_{1}\otimes
y_{2})\otimes (z_{1}\otimes z_{2})))\\\ &=&\mu ((x_{1}\otimes
x_{2})\otimes (m( y_{1}\otimes y_{2}\otimes z_{1})\otimes
z_{2}))+\alpha\ \mu ((x_{1}\otimes x_{2})\otimes (y_{1}\otimes m
(y_{2}\otimes z_{1}\otimes z_{2}))\\
\ &=&m(x_{1}\otimes x_{2}\otimes
m(y_{1}\otimes y_{2}\otimes z_{1}))\otimes z_{2} +\alpha\
x_{1}\otimes m( x_{2}\otimes m(y_{1}\otimes y_{2}\otimes
z_{1})\otimes z_{2}) \\
\  & &+\alpha\  m(x_{1}\otimes x_{2}\otimes
y_{1})\otimes m( y_{2}\otimes z_{1}\otimes z_{2}) +\alpha ^{2}\
x_{1}\otimes m(x_{2}\otimes y_{1}\otimes  m(y_{2}\otimes
z_{1}\otimes z_{2}))
\end{eqnarray*}

The multiplication $\mu$ is of associative type if $A_{1}+\lambda\ A_{2}=0$, that is
\begin{eqnarray*}
A_{1}+\lambda A_{2}&=&[m(m(x_{1}\otimes x_{2}\otimes y_{1})\otimes
y_{2}\otimes z_{1})+\alpha\ m(x_{1}\otimes m(x_{2}\otimes
y_{1}\otimes y_{2})\otimes z_{1}) \\ \ && +\lambda\ m(x_{1}\otimes x_{2}\otimes m(
y_{1}\otimes y_{2}\otimes z_{1}))]\otimes z_{2} +\alpha\
x_{1}\otimes [\alpha\ m(m(x_{2}\otimes y_{1}\otimes  y_{2})\otimes
z_{1}\otimes z_{2})\\\ && +\lambda\ m(x_{2}\otimes m(y_{1}\otimes  y_{2}\otimes
z_{1})\otimes z_{2})+\alpha\lambda \ m(x_{2}\otimes y_{1}\otimes
m(y_{2}\otimes z_{1}\otimes z_{2}))]\\\ &&+ \alpha(1+\lambda)\  m(x_{1}\otimes
x_{2}\otimes y_{1})\otimes m(y_{2}\otimes z_{1}\otimes z_{2})
\end{eqnarray*}

 If $m$ is a ternary operation which defines a partially associative ternary
 algebra of a given type, then $A_{1}+\lambda A_{2}=0$ if $\alpha(1+\lambda)=0$ and the coefficients
 $(1,\alpha,\lambda)$,
$(\alpha^2,\lambda\alpha,\lambda\alpha^2)$ are proportional.

The first condition is satisfied when $\alpha=0$ or $\lambda =-1$. The case $\alpha=0$ is impossible.

If $\lambda =-1$, the coefficients
 $(1,\alpha,-1)$ and
$(\alpha^2,-\alpha,-\alpha^2)$ should be proportional. This is possible only
over  $\mathbb{C}$ with $\alpha=\pm i$. The associativity condition
needed must be of one of the following  forms
\begin{equation*}m(m(x_{1}\otimes x_{2}\otimes x_{3})\otimes x_{4}\otimes x_{5})+
i\ m(x_{1}\otimes m(x_{2}\otimes x_{3}\otimes x_{4})\otimes x_{5})-
m(x_{1}\otimes x_{2}\otimes m(x_{3}\otimes x_{4}\otimes x_{5}))=0
\end{equation*}
\begin{equation*}m(m(x_{1}\otimes x_{2}\otimes x_{3})\otimes x_{4}\otimes x_{5})-
i\ m(x_{1}\otimes m(x_{2}\otimes x_{3}\otimes x_{4})\otimes x_{5})-
 m(x_{1}\otimes x_{2}\otimes m( x_{3}\otimes x_{4}\otimes x_{5}))=0
\end{equation*}
In the both cases one may construct a cohomology of ternary algebras
according to Takhtajan's construction and  using the Hochschild complex of associative binary multiplication.

Therefore, we have the following proposition :

\begin{proposition}
It is impossible to construct, using Takhtajan's construction, a
 cohomology of ternary algebras  $(V,m)$ which are
partially  associative (resp. alternate partially associative) starting from a complex of binary algebra of
associative type.

\end{proposition}

\begin{remark}
If the ternary algebra $m$ is totally associative, then the
corresponding binary algebra is of associative type if
\begin{eqnarray*}
  1+\alpha +\lambda =0 \\
  \alpha ^2 +\lambda \alpha+\lambda \alpha^2 =0 \\
  \alpha (1+\lambda )=0
\end{eqnarray*}
Which implies that $\alpha =0$ and $\lambda =-1$.

Therefore, using Takhtajan procedure, we can construct a cohomology of totally associative
ternary algebras $(V,m)$ with a binary multiplication $\mu$ defined on $W=V \otimes V$
  by
\begin{equation*}
\mu ((x_1 \otimes x_2)\otimes (y_1 \otimes y_2))=m(x_1 \otimes x_2\otimes y_1)
 \otimes y_2 \quad
\forall x_1 \otimes x_2,y_1 \otimes y_2\in W
\end{equation*}

Let $\varphi : V^{\otimes 2p-1}\rightarrow V$ be a $p$-cochain of the totally
associative ternary algebra $(V,m)$. We set
$\Delta \varphi (x_1 \otimes \cdots \otimes x_{2p})=\varphi (x_1 \otimes \cdots \otimes x_{2p-1})\otimes  x_{2p}$.
 It turns out that in this case, we recover the coboundary map of the weak totally associative ternary
 algebras discussed in Section 3.

 Indeed, using the Hochschild coboundary of the binary associative algebra $(W,\mu)$, we have
 \begin{eqnarray*}
 d^p \Delta \varphi (x_1 \otimes \cdots \otimes x_{2p+2})&=&
 \mu (x_1 \otimes x_2 \otimes \Delta \varphi (x_3 \otimes \cdots \otimes x_{2p+2}))\\& &
 +\sum_{i=1}^{p}{\Delta \varphi (x_1 \otimes \cdots \otimes\mu (x_{2i-1}\otimes x_{2i}
 \otimes x_{2i+1}\otimes x_{2i+2})\otimes\cdots \otimes x_{2p+2})}\\
 && + (-1)^p \mu (\Delta \varphi (x_1 \otimes \cdots \otimes x_{2p})\otimes
x_{2p+1}\otimes x_{2p+2})\\
 \ &=&
 m (x_1 \otimes x_2 \otimes  \varphi (x_3 \otimes \cdots \otimes x_{2p+1}))\otimes x_{2p+2}\\& &
 +\sum_{i=1}^{p}{\varphi (x_1 \otimes \cdots \otimes m (x_{2i-1}\otimes x_{2i}
 \otimes x_{2i+1})\otimes\cdots \otimes x_{2p+1})}\otimes x_{2p+2}\\
 && + (-1)^p m (\varphi (x_1 \otimes \cdots \otimes x_{2p-1})\otimes
 x_{2p}\otimes x_{2p+1})\otimes x_{2p+2}
 \end{eqnarray*}

 Then, we set
 \begin{eqnarray*}
 \delta^p  \varphi (x_1 \otimes \cdots \otimes x_{2p+1})&=&
 m (x_1 \otimes x_2 \otimes  \varphi (x_3 \otimes \cdots \otimes x_{2p+1}))
\\& &
 +\sum_{i=1}^{p}{\varphi (x_1 \otimes \cdots \otimes m (x_{2i-1}\otimes x_{2i}
 \otimes x_{2i+1})\otimes\cdots \otimes x_{2p+1})}\\
 && + (-1)^p m (\varphi (x_1 \otimes \cdots \otimes x_{2p-1})\otimes
 x_{2p}\otimes x_{2p+1})
 \end{eqnarray*}
 and recover the cohomology of weak totally associative ternary algebras defined above.
\end{remark}

\section{On Deformation Cohomology of Skew-Associative Algebras}
In this section, we show that the 1-cohomology and 2-cohomology
guided by 1-parameter formal deformations cannot be extended to a
3-cohomology. Therefore the operad of skew-associative binary
algebras is not Koszul.
\begin{definition}
A skew-associative binary algebra is given by a $\mathbb{K}$-vector
space $V$\ and a bilinear multiplication $\mu$ satisfying, for every $x_{1},x_{2},x_{3}\in V$%
,
\begin{equation}
\mu(\mu(x_{1}\otimes x_{2})\otimes x_{3})=-\mu(x_{1}\otimes \mu(x_{2}\otimes x_{3}))
\end{equation}

\end{definition}

The formal deformation theory leads to the following $1$-coboundary
and $2$-coboundary operators for a  cohomology of skew-associative
binary algebra $\mathcal{A}=(V,\mu)$ adapted to formal deformation theory.
The  $1$-coboundary operator of $\mathcal{A}$ is  the map
\begin{equation*}
\delta ^{1}: \mathcal{C} ^{0}({\mathcal{A} ,\mathcal{A} }) \longrightarrow \mathcal{C}
^{1}({\mathcal{A} ,\mathcal{A} }), \quad f \longmapsto \delta ^{1}f
\end{equation*}
defined by
\begin{equation*}
\delta ^{1}f(x_{1}\otimes x_{2}) =f(\mu(x_{1}\otimes x_{2}))-\mu(f(x_{1})\otimes x_{2})-
\mu(x_{1}\otimes f(x_{2}))
\end{equation*}

The  $2$-coboundary operator of ${\mathcal{A}}$ the map
\begin{equation*}
\delta^{2}: \mathcal{C}^{1}({\mathcal{A} ,\mathcal{A} }) \longrightarrow \mathcal{C}^{2}({\mathcal{A} ,\mathcal{A}}), \quad \varphi \longmapsto \delta ^{2}\varphi
\end{equation*}
defined by
\begin{align*}
\delta^{2}\varphi (x_{1}\otimes x_{2}\otimes x_{3}) &= \mu(\varphi (x_{1}\otimes
x_{2})\otimes x_{3}))+\mu(x_{1}\otimes \varphi (x_{2}\otimes x_{3})) \\
&+\varphi (\mu(x_{1}\otimes x_{2})\otimes x_{3}) +\varphi (x_{1}\otimes \mu(x_{2}\otimes x_{3}))
\end{align*}

One may characterize the operator $\delta^{2}$ using the following skew-associator map
\begin{equation*}
\circ : \mathcal{C}^{r}({\mathcal{A} ,\mathcal{A} }) \times \mathcal{C}^{s}({\mathcal{A} ,\mathcal{A} }) \longrightarrow \mathcal{C}^{r+s}({%
\mathcal{A} ,\mathcal{A} }), \quad (f,g) \longmapsto f\circ g
\end{equation*}
defined by
\begin{equation*}
f\circ g (x_1\otimes\cdots\otimes x_{r+s})=\sum_{i=0}^{r-1}{f (
x_1\otimes\cdots\otimes g (x_{i+1}\otimes\cdots\otimes x_{i+s}),
\cdots\otimes x_{r+s})}
\end{equation*}
We have
\begin{equation*}\delta^{2}\varphi=\mu \circ \varphi +\varphi \circ \mu
\end{equation*}
 Note also that
$\delta ^{2}\circ \delta ^{1}=0.$

\begin{proposition}
A  $3$-coboundary operator   extending the maps
$\delta^{1}$ and $\delta^{2}$ to a complex of for  skew-associative
binary algebras doesn't exist.
\end{proposition}
\begin{proof}
We set the following general form of $3$-coboundary operator
\begin{eqnarray*} \delta ^3f (x_1\otimes x_2\otimes x_3\otimes x_4)=a_1 \mu
(x_1\otimes f(x_2\otimes  x_3\otimes x_4))+ a_2 f(\mu (x_1\otimes x_2)\otimes x_3\otimes x_4)+\\ a_3 f(x_1\otimes
\mu (x_2\otimes x_3)\otimes x_4 )+ a_4 f(x_1\otimes x_2\otimes
\mu(x_3\otimes x_4 ))+a_5 \mu(f((x_1\otimes x_2\otimes x_3)\otimes
x_4)
\end{eqnarray*}

We consider a 3-cochain $f$, that is a map $f:V^{\otimes 3}\rightarrow V$, and a $2$-cochain
 $g$, that is a map $f:V^{\otimes 2}\rightarrow V$. We compute $\delta^3 ( \delta^2 g)
(x_1\otimes x_2\otimes x_3\otimes x_4)$ and substitute $\mu
(y_1\otimes \mu(y_2\otimes y_3))$ by $-\mu(\mu(y_1\otimes
y_2)\otimes y_3)$. Then, we obtain
\begin{eqnarray*}
(a_3-a_4) f(x_1\otimes \mu (\mu (x_2\otimes x_3)\otimes x_4 )+
(a_2+a_4) f(\mu (x_1\otimes x_2)\otimes \mu (x_3\otimes x_4 ))+\\
(a_2 -a_3)f(\mu (\mu (x_1\otimes x_2)\otimes x_3)\otimes x_4 ) +
(a_1 +a_4) \mu (x_1\otimes f(x_2\otimes \mu (x_3\otimes x_4 )))+\\
(a_1+a_3 )\mu (x_1\otimes f(\mu (x_2\otimes x_3)\otimes x_4 ))+
(a_3 +a_5) \mu (f(x_1\otimes \mu (x_2\otimes x_3))\otimes x_4 )+\\
(a_2 +a_5) \mu (f(\mu (x_1\otimes x_2)\otimes x_3)\otimes x_4 )+
(a_2 -a_1) \mu (\mu (x_1\otimes x_2]\otimes f(x_3\otimes x_4 ))+\\
(a_5-a_1) \mu (\mu (x_1\otimes f(x_2\otimes x_3))\otimes x_4 )+ (a_5
-a_4) \mu (\mu (f(x_1\otimes x_2)\otimes x_3)\otimes x_4 )=0
\end{eqnarray*}
The equation is satisfied for all $x_1,x_2,x_3,x_4\in
V$ if and only if  $a_{i_{\{i=1,...,5\}}}$ are all equal to 0.
\end{proof}
\begin{corollary}
A deformation cohomology for skew-associative
binary algebras doesn't exist. Then the operad  of skew-associative binary
algebras is not Koszul.
\end{corollary}

\end{document}